\documentclass[10pt]{article}

\usepackage{graphicx}
\usepackage{amsmath}
\usepackage{amssymb}
\usepackage{amsthm}
\usepackage{float}

\usepackage{color}
\usepackage{multirow}

\newtheorem{theorem}{Theorem}[section]
\newtheorem{lemma}{Lemma}[section]
\newtheorem{remark}{Remark}[section]
\newtheorem{scheme}{Scheme}
\newtheorem{example}{Example}
\newtheorem{hypothesis}{Hypothesis}

\begin{document}

\title{A weak Local Linearization scheme for stochastic differential
equations with multiplicative noise}

\author{J.C. Jimenez \thanks{%
Instituto de Cibern\'{e}tica, Matem\'{a}tica y F\'{i}sica. La
Habana, Cuba. email: jcarlos@icimaf.cu},
C. Mora\thanks{%
Departamento de Ingenier\'{\i}a Matem\'{a}tica and CI$^2$MA,
Universidad de Concepci\'{o}n, Chile. email: cmora@ing-mat.udec.cl},
M. Selva\thanks{%
Departamento de Ingener\'{\i}a Matem\'{a}tica, Universidad de
Concepci\'{o}n, Chile. email: selva@ing-mat.udec.cl}}

\maketitle

\begin{abstract}
In this paper, a weak Local Linearization scheme for Stochastic
Differential Equations (SDEs) with multiplicative noise is
introduced. First, for a time discretization, the solution of the
SDE is locally approximated by the solution of the piecewise linear
SDE that results from the Local Linearization strategy. The weak
numerical scheme is then defined as a sequence of random vectors
whose first moments coincide with those of the piecewise linear SDE
on the time discretization. The rate of convergence is derived and
numerical simulations are presented for illustrating the performance
of the scheme.
\end{abstract}

\section{Introduction}

During 30 years the class of local linearization integrators has
been developed for different types of deterministic and random
differential equations. The essential principle of such integration
methods is the piecewise linearization of the given differential
equation to obtain consecutive linear equations that are explicitly
solved at each time step. This general approach has worked well for
the classes of ordinary, delay, random and stochastic differential
equations with additive noise. Key element of such success is the
use of explicit solutions or suitable approximations for the
resulting linear differential equations. Precisely, the absence of
explicit solution or adequate approximation for linear Stochastic
Differential Equations (SDEs) with multiplicative noise is the main
reason of the limited application of the Local Linearization
approach to nonlinear SDEs with multiplicative noise. For these
equations, the available local linearization integrators are of two
types: the introduced in \cite{Biscay 1996} for scalar equations and
the considered in \cite{Mora2005, Shoji11, Stramer99}. The former
uses the explicit solution of the scalar linear equations with
multiplicative noise, while the latter employs the solution of the
linear equation with additive noise that locally approximates the
nonlinear equation.

Directly related to the development of the local linearization
integrators is the concept of Local Linear approximations (see,
e.g., \cite{Jimenez02 SAA,Jimenez05 AMC,Jimenez03 IJC}). These
approximations to the solution of the differential equations are
defined as the continuous time solution of the piecewise linear
equations associated to the Local Linearization method. These
continuous approximations have played a fundamental role for
studying the convergence, stability and dynamics of the local
linearization integrators for all the classes of differential
equations mentioned above with the exception of the SDEs with
multiplicative noise. For this last class of equations, the Local
Linear approximations have only been used for constructing piecewise
approximations to the mean and variance of the states in the
framework of continuous-discrete filtering problems (see
\cite{Jimenez03 IJC}).

The purpose of this work is to construct a weak Local Linearization
integrator for SDEs with multiplicative noise based on suitable weak
approximation to the solution of piecewise linear SDEs with
multiplicative noise. For this, we cross two ideas: 1) as in
\cite{Jimenez03 IJC}, the use of the Local Linear approximations for
constructing piecewise approximations to the mean and variance of
the SDEs with multiplicative noise; and 2) as in \cite
{Carbonell06}, at each integration step, the generation of a random
vector with the mean and variance of the Local Linear approximation
at this integration time. For implementing this, new formulas
recently obtained in \cite{Jimenez 2012} for the mean and variance
of the solution of linear SDEs with multiplicative noise are used,
which are computationally more efficient than those formerly
proposed in \cite {Jimenez02 SCL,Jimenez03 IJC}. Notice that this
integration approach is conceptually different to that usually
employed for designing weak integrators for SDEs. Typically, these
integrators are derived from a truncated Ito-Taylor expansion of the
equation's solution at each integration step, and include the
generation of random variables with moments equal to those of the
involved multiple Ito integrals \cite{Kloeden 1995,Milstein04}.

The paper is organized as follows. After some basic notations in
Section 2, the new Local Linearization integrator is introduced in
Section 3. Its rate of convergence is derived in Section 4 and, in
the last section, numerical simulations are presented in order to
illustrate the performance of the numerical integrator.

\section{Basic notations}

Let us consider the SDE with multiplicative noise
\begin{equation}
X_{t}=X_{t_{0}}+\int_{t_{0}}^{t}f\left( s,X_{s}\right)
ds+\sum_{k=1}^{m}\int_{t_{0}}^{t}g^{k}\left( s,X_{s}\right)
dW_{s}^{k},\quad \quad \forall t\in \lbrack t_{0},T],
\label{eq:1.1}
\end{equation}%
where $f,g^{k}:\left[ t_{0},T\right] \times
\mathbb{R}^{d}\rightarrow \mathbb{R}^{d}$ are smooth functions,
$W^{1},\ldots ,W^{m}$ are independent
Wiener processes on a filtered complete probability space $\left( \Omega ,%
\mathfrak{F},\left( \mathfrak{F}_{t}\right) _{t\geq
t_{0}},\mathbb{P}\right) $, and $X_{t}$ is an adapted
$\mathbb{R}^{d}$-valued stochastic process. In addition, let us
assume the usual conditions for the existence and uniqueness of a
weak solution of (\ref{eq:1.1}) with bounded moments (see, e.g.,
\cite{Kloeden 1995}).

Throughout this paper, we consider the time discretization
$t_{0}=\tau _{0}<\tau _{1}<\dots <\tau _{N}=T$ with $\tau
_{n+1}-\tau _{n}\leq \Delta $ for all $n=0,\ldots ,N-1$ and $\Delta
>0$. We use the same symbol $K\left( \cdot \right) $ (resp., $K$)
for different positive increasing functions (resp., positive real
numbers) having the common property to be independent of $\left(
\tau _{k}\right) _{k=0,\ldots ,N}$. Moreover, $A^{\mathbf{\top }} $
stands for the transpose of the matrix $A$, and $\left\vert \cdot
\right\vert $ denotes the Euclidean norm for vectors or the
Frobenious norm
for matrices. By $\mathcal{C}_{P}^{\ell }\left( \mathbb{R}^{d},\mathbb{R}%
\right) $ we mean the collection of all $\ell $-times continuously
differentiable functions $g:\mathbb{R}^{d}\rightarrow \mathbb{R}$ such that $%
g$ and all its partial derivatives of orders $1,2,\ldots ,\ell $
have at most polynomial growth.

\section{Numerical method}

Suppose that $z_{n}\approx X_{\tau _{n}}$ with $n=0,\ldots ,N-1$. Set $%
g^{0}=f$. Taking the first-order Taylor expansion of $g^{k}$ yields
\begin{equation*}
g^{k}\left( t,x\right) \approx g^{k}\left( \tau _{n},z_{n}\right) +\frac{%
\partial g^{k}}{\partial x}\left( \tau _{n},z_{n}\right) \left(
x-z_{n}\right) +\frac{\partial g^{k}}{\partial t}\left( \tau
_{n},z_{n}\right) \left( t-\tau _{n}\right)
\end{equation*}%
whenever $x\approx z_{n}$ and $t\approx \tau _{n}$. Therefore
\begin{equation*}
X_{t}\approx z_{n}+\sum_{k=0}^{m}\int_{\tau _{n}}^{t}\left(
B_{n}^{k}X_{s}+b_{n}^{k}\left( s\right) \right) dW_{s}^{k}\quad
\quad \forall t\in \left[ \tau _{n},\tau _{n+1}\right] ,
\end{equation*}%
with $W_{s}^{0}=s$, $B_{n}^{k}=\dfrac{\partial g^{k}}{\partial
x}\left( \tau _{n},z_{n}\right) $ and
\begin{equation}
b_{n}^{k}\left( s\right) =g^{k}\left( \tau _{n},z_{n}\right)
-\frac{\partial
g^{k}}{\partial x}\left( \tau _{n},z_{n}\right) z_{n}+\frac{\partial g^{k}}{%
\partial t}\left( \tau _{n},z_{n}\right) \left( s-\tau _{n}\right) .
\label{eq:1.9}
\end{equation}%
This follows that, for all $t \in \left[ \tau
_{n},\tau_{n+1}\right]$, $X_t$ can be approximated by
\begin{equation}\label{eq:1.2}
Y_{t}=z_{n}+\sum_{k=0}^{m}\int_{\tau _{n}}^{t}
\left(B_{n}^{k}Y_{s}+b_{n}^{k}\left(s\right) \right) dW_{s}^{k},\hspace{2cm}%
\forall t\in \left[ \tau _{n},\tau _{n+1}\right],
\end{equation}%
which is the first order Local Linear approximation of $X_{t}$ used
in \cite{Jimenez03 IJC}. Hence, $\mathbb{E}\phi \left(
X_{\tau_{n+1}}\right) \approx \mathbb{E}\phi \left( Y_{\tau
_{n+1}}\right)$ for any smooth function $\phi$, and so $X_{\tau
_{n+1}}$ might be weakly approximated by a random variable $z_{n+1}$
such that the first moments of $z_{n+1}-z_{n}$ be similar to those
of $Y_{\tau _{n+1}}-z_{n}$. This leads us to the following Local
Linearization scheme.
\begin{scheme}
\label{scheme:MM1} Let $\eta _{0}^{1},\ldots ,\eta _{0}^{m},\cdots
,\eta _{N-1}^{1},\ldots ,\eta _{N-1}^{m}$ be i.i.d. symmetric random
variables having variance $1$ and finite moments of any order. For a
given $z_{0}$, we define recursively $\left( z_{n}\right)
_{n=0,\ldots ,N}$ by
\begin{equation}
z_{n+1}=\mu _{n}\left( \tau _{n+1}\right) +\sqrt{\sigma _{n}\left(
\tau _{n+1}\right) -\mu _{n}\left( \tau _{n+1}\right) \mu
_{n}^{\intercal }\left( \tau _{n+1}\right) }\,\eta _{n},
\label{eq:1.6}
\end{equation}%
where $\eta _{n}=\left( \eta _{n}^{1},\ldots ,\eta _{n}^{m}\right)
^{\top }$ and $\mu _{n}\left( t\right) $, $\sigma _{n}\left(
t\right) $ satisfy the linear differential equations
\begin{equation}
\mu _{n}\left( t\right) =z_{n}+\int_{\tau _{n}}^{t}\left(
B_{n}^{0}\,\mu _{n}\left( s\right) +b_{n}^{0}\left( s\right) \right)
ds\quad \quad \forall t\in \left[ \tau _{n},\tau _{n+1}\right] ,
\label{eq:1.3}
\end{equation}%
\begin{equation}
\sigma _{n}\left( t\right) =z_{n}z_{n}^{\top }+\int_{\tau _{n}}^{t}\mathcal{L%
}_{n}\left( s,\sigma _{n}\left( s\right) \right) ds\quad \quad \forall t\in %
\left[ \tau _{n},\tau _{n+1}\right] .  \label{eq:1.4}
\end{equation}%
Here
\begin{align*}
\mathcal{L}_{n}\left( s,\sigma \right) & =\sigma \left(
B_{n}^{0}\right) ^{\top }+B_{n}^{0}\,\sigma ^{\intercal }+\mu
_{n}\left( s\right) (b_{n}^{0}\left( s\right) )^{\top
}+b_{n}^{0}\left( s\right) \mu
_{n}^{\intercal }\left( s\right) \\
& \quad +\sum_{k=1}^{m}\left( B_{n}^{k}\,\sigma \left(
B_{n}^{k}\right) ^{\top }+B_{n}^{k}\,\mu _{n}\left( s\right)
(b_{n}^{k}\left( s\right) )^{\top }+b_{n}^{k}\left( s\right) \mu
_{n}^{\intercal }\left( s\right) \left( B_{n}^{k}\right) ^{\top
}+b_{n}^{k}\left( s\right) (b_{n}^{k}\left( s\right) )^{\top
}\right) .
\end{align*}
\end{scheme}

\begin{remark}
From (\ref{eq:1.3}) it follows that $\mu _{n}\left( \tau
_{n+1}\right) $ is the expected valued of $Y_{\tau _{n+1}}$ given
$Y_{\tau _{n}}=z_{n}$. Moreover, (\ref{eq:1.4}) implies
\begin{equation*}
\sigma _{n}\left( \tau _{n+1}\right) =\mathbb{E}\left( Y_{\tau
_{n+1}}Y_{\tau _{n+1}}^{\top }\diagup Y_{\tau _{n}}=z_{n}\right) .
\end{equation*}
\end{remark}

\begin{remark}
By construction, Scheme \ref{scheme:MM1} preserves the mean-square
stability property that the solution of the linear equation $dX_{t}=
\sum_{k=0}^{m}(B^{k}X_{t}+b^{k,1}t+b^{k,0})dW_{t}^{k}$ might have.
For instance, if the trivial solution of the homogenous equation $
dX_{t}=\sum_{k=0}^{m}B^{k}X_{t}dW_{t}^{k}$ is mean-square
asymptotically stable, Scheme \ref{scheme:MM1} inherits this
property.
\end{remark}

\begin{remark}
A key point in the implementation of Scheme \ref{scheme:MM1} is the
evaluation of just one matrix exponential for computing $\mu
_{n}\left( \tau _{n+1}\right) $ and $\sigma _{n}\left( \tau
_{n+1}\right) $ at each time step. Indeed, from Theorem 2 in
\cite{Jimenez 2012},
\begin{equation}
\mu _{n}\left( \tau _{n+1}\right) =z_{n}+\mathcal{L}_{2}e^{\mathcal{M}%
_{n}(\tau _{n+1}-\tau _{n})}u_{n}  \label{eq:1.7}
\end{equation}%
and%
\begin{equation}
vec(\sigma _{n}\left( \tau _{n+1}\right) )=\mathcal{L}_{1}e^{\mathcal{M}%
_{n}(\tau _{n+1}-\tau _{n})}u_{n},  \label{eq:1.8}
\end{equation}%
where the matrices $\mathcal{M}_{n}$, $\mathcal{L}_{1}$,
$\mathcal{L}_{2}$\ and the vector $u_{n}$ are given by
\begin{equation*}
\mathcal{M}_{n}=\left[
\begin{array}{cccccc}
\mathcal{A} & \mathcal{B}_{5} & \mathcal{B}_{4} & \mathcal{B}_{3} & \mathcal{%
B}_{2} & \mathcal{B}_{1} \\
\mathbf{0} & \mathcal{C} & \mathcal{I}_{d+2} & 0 & 0 & 0 \\
\mathbf{0} & 0 & \mathcal{C} & 0 & 0 & 0 \\
\mathbf{0} & 0 & 0 & 0 & 2 & 0 \\
\mathbf{0} & 0 & 0 & 0 & 0 & 1 \\
\mathbf{0} & \mathbf{0} & \mathbf{0} & 0 & 0 & 0%
\end{array}%
\right] ,\text{ }u_{n}=\left[
\begin{array}{c}
vec(z_{n}z_{n}^{\intercal }) \\
0 \\
r \\
0 \\
0 \\
1%
\end{array}%
\right] \in \mathbb{R}^{d^{2}+2d+7},
\end{equation*}%
$\mathcal{L}_{2}=\left[
\begin{array}{ccc}
0_{d\times (d^{2}+d+2)} & \mathcal{I}_{d} & 0_{d\times 5}%
\end{array}%
\right] $ and $\mathcal{L}_{1}=\left[
\begin{array}{cc}
\mathcal{I}_{d^{2}} & 0_{d^{2}\times (2d+7)}%
\end{array}%
\right] $, with matrices $\mathcal{A}$, $\mathcal{B}_{i}$,
$\mathcal{C}$ and $r$ defined by
\begin{equation*}
\mathcal{A}=B_{n}^{0}\oplus
B_{n}^{0}+\sum\limits_{k=1}^{m}B_{n}^{k}\otimes
(B_{n}^{k})^{\intercal },\text{ \ \ \ }\mathcal{C}=\left[
\begin{array}{ccc}
B_{n}^{0} & b_{n}^{0,1} & B_{n}^{0}z_{n}+b_{n}^{0,0} \\
0 & 0 & 1 \\
0 & 0 & 0%
\end{array}%
\right] ,\text{\ \ \ \ \ }r=\left[
\begin{array}{c}
\mathbf{0}_{(d+1)\times 1} \\
1%
\end{array}%
\right] ,
\end{equation*}%
$\mathcal{B}_{1}=vec(\beta _{1})+\beta _{4}z_{n}$, $\mathcal{B}%
_{2}=vec(\beta _{2})+\beta _{5}z_{n}$, $\mathcal{B}_{3}=vec(\beta _{3})$, $%
\mathcal{B}_{4}=\beta _{4}\mathcal{L}$, and $\mathcal{B}_{5}=\beta _{5}%
\mathcal{L}$. Here $\mathcal{L}=\left[
\begin{array}{ll}
\mathcal{I}_{d} & 0_{d\times 2}%
\end{array}%
\right] $ and
\begin{gather*}
\beta _{1}=\sum\limits_{k=1}^{m}b_{n}^{k,0}(b_{n}^{k,0})^{\intercal
},\text{ \ \ \ \ }\beta
_{2}=\sum\limits_{k=1}^{m}b_{n}^{k,0}(b_{n}^{k,1})^{\intercal
}+b_{n}^{k,1}(b_{n}^{k,0})^{\intercal },\text{ \ \ \ \ \ }\beta
_{3}=\sum\limits_{k=1}^{m}b_{n}^{k,1}(b_{n}^{k,1})^{\intercal },\text{ } \\
\text{\ \ \ }\beta _{4}=b_{n}^{0,0}\oplus
b_{n}^{0,0}+\sum\limits_{k=1}^{m}b_{n}^{k,0}\otimes
B_{n}^{k}+B_{n}^{k}\otimes b_{n}^{k,0},\text{ \ \ }\beta
_{5}=b_{n}^{0,1}\oplus
b_{n}^{0,1}+\sum\limits_{k=1}^{m}b_{n}^{k,1}\otimes
B_{n}^{k}+B_{n}^{k}\otimes b_{n}^{k,1},
\end{gather*}%
being $b_{n}^{k,0}$ and $b_{n}^{k,1}$ defined via (\ref{eq:1.9}) as
$ b_{n}^{k,0}+b_{n}^{k,1}\left( s-\tau _{n}\right) =b_{n}^{k}\left(
s\right) $. The symbols $vec$, $\oplus $ and $\otimes $ denote the
vectorization operator, the Kronecker sum and the Kronecker product,
respectively. $\mathcal{I}_{d}$ is the $d-$dimensional identity
matrix. The matrix exponential in (\ref{eq:1.7}) and (\ref{eq:1.8})
can be efficiently computed via the Pad\'{e} method with scaling and
squaring strategy or via the Krylov subspace method in the case of
large system of SDEs (see, e.g., \cite{Moler03}). For autonomous
equations or for equations with additive noise, the exponential
matrix in (\ref{eq:1.7}) and (\ref{eq:1.8}) reduces to simpler forms
\cite{Jimenez 2012}.
\end{remark}

\begin{remark}
For SDEs with additive noise, Scheme \ref{scheme:MM1} reduces to the
weak order-$1$ Local Linearization scheme introduced in
\cite{Carbonell06}.
\end{remark}

\section{Rate of convergence}

Next theorem establishes the linear rate of weak convergence of
Scheme \ref{scheme:MM1} when the drift and diffusion coefficients
are smooth enough.

\begin{hypothesis}
\label{hyp:Hip1} For any $k=0,\ldots ,m$ we have $g^{k}\in \mathcal{C}%
_{P}^{4}\left( \left[ t_{0},T\right] \times \mathbb{R}^{d},\mathbb{R}%
^{d}\right) $. Moreover,
\begin{equation}\label{eq:1.5}
\left\vert g^{k}\left(t,x\right)\right\vert \leq K\left(1+\left\vert
x\right\vert \right) \qquad \mbox{ and } \qquad \left\vert
\dfrac{\partial g^{k}}{\partial t}\left(t,x\right)\right\vert +
\left\vert \dfrac{\partial g^{k}}{\partial x}\left(t,x\right)
\right\vert \leq K
\end{equation}%
for all $t\in \left[ t_{0},T\right] $ and $x\in \mathbb{R}^{d}$.
\end{hypothesis}

\begin{theorem}
\label{th:convScheme} In addition to Hypothesis \ref{hyp:Hip1},
suppose that $X_{t_{0}}$ has finite moments of any order and that
for all $\phi \in \mathcal{C}_{P}^{4}\left(
\mathbb{R}^{d},\mathbb{R}\right) $,
\begin{equation*}
\left\vert \mathbb{E}\phi \left( X_{t_{0}}\right) -\mathbb{E}\phi
\left( z_{0}\right) \right\vert \leq K\Delta .
\end{equation*}%
Then, for all $\phi \in \mathcal{C}_{P}^{4}\left( \mathbb{R}^{d},\mathbb{R}%
\right) $,
\begin{equation*}
\left\vert \mathbb{E}\phi \left( X_{T}\right) -\mathbb{E}\phi \left(
z_{N}\right) \right\vert \leq K\left( T\right) \Delta ,
\end{equation*}%
where $z_{N}$ is given by Scheme \ref{scheme:MM1}.
\end{theorem}

Theorem \ref{th:convScheme} is a straightforward result of Theorem
14.5.2 in \cite{Kloeden 1995} and the two following Lemmata.

\begin{lemma}
\label{lem:CotaCrecimiento} Under the assumptions of Theorem \ref%
{th:convScheme}, for any $q\geq 1$ we have
\begin{equation}
\mathbb{E}\left( \max_{n=0,\ldots ,N}\left\vert z_{n}\right\vert
^{2q}\right) \leq K\left( T\right) \left( 1+\mathbb{E}\left(
\left\vert z_{0}\right\vert ^{2q}\right) \right)  \label{eq:2.6}
\end{equation}%
and
\begin{equation}
\mathbb{E}\left( \left\vert z_{n+1}-z_{n}\right\vert ^{2q}\diagup \mathfrak{F%
}_{\tau _{n}}\right) \leq K\left( T\right) \left( \tau _{k+1}-\tau
_{k}\right) ^{q}\left( 1+\left\vert z_{n}\right\vert ^{2q}\right)
\label{eq:2.7}
\end{equation}%
for all $n=0,\ldots ,N-1$.
\end{lemma}

\begin{proof}
From Hypothesis \ref{hyp:Hip1} it follows that $\left\vert
B_{n}^{k}\right\vert \leq K$ and
\begin{equation}
\left\vert b_{n}^{k}\left( s\right) \right\vert \leq K\left(
T\right) \left( 1+\left\vert z_{n}\right\vert \right)
\label{eq:2.17}
\end{equation}%
for all $n=0,\ldots ,N-1$, $k=0,\ldots ,m$ and $s\in \left[ \tau
_{n},\tau _{n+1}\right] $. Then, combining Gronwall's lemma with
(\ref{eq:1.3}) gives
\begin{equation}
\left\vert \mu _{n}\left( s\right) \right\vert \leq K\left( T\right)
\left( 1+\left\vert z_{n}\right\vert \right) \quad \quad \quad
\forall s\in \left[ \tau _{n},\tau _{n+1}\right] .  \label{eq:2.2}
\end{equation}%
Since $\left\vert x\,y^{\top }\right\vert =\left\vert x\right\vert
\left\vert y\right\vert $ for any $x,y\in \mathbb{R}^{d}$,
(\ref{eq:2.17}) and (\ref{eq:2.2}) lead to
\begin{equation}
\left\vert \mathcal{L}_{n}\left( s,\sigma \right) \right\vert \leq
K\left\vert \sigma \right\vert +K\left( T\right) \left( 1+\left\vert
z_{n}\right\vert ^{2}\right) \quad \quad \quad \forall s\in \left[
\tau _{n},\tau _{n+1}\right] ,  \label{eq:2.1}
\end{equation}%
where $n=0,\ldots ,N-1$ and $\mathcal{L}_{n}$ is as in
(\ref{eq:1.4}). Using Gronwall's lemma, (\ref{eq:1.4}) and
(\ref{eq:2.1}) we deduce that
\begin{equation}
\left\vert \sigma _{n}\left( s\right) \right\vert \leq K\left(
T\right) \left( 1+\left\vert z_{n}\right\vert ^{2}\right) \quad
\quad \quad \forall s\in \left[ \tau _{n},\tau _{n+1}\right] .
\label{eq:2.3}
\end{equation}%
Decomposing
\begin{equation*}
\widetilde{\sigma }_{n}\left( t\right) :=\sigma _{n}\left( t\right)
-\mu _{n}\left( t\right) \mu _{n}^{\intercal }\left( t\right)
\end{equation*}%
as $\sigma _{n}\left( t\right) -z_{n}z_{n}^{\top }-z_{n}\left( \mu
_{n}\left( t\right) -z_{n}\right) ^{\top }-\left( \mu _{n}\left(
t\right) -z_{n}\right) z_{n}^{\top }-\left( \mu _{n}\left( t\right)
-z_{n}\right) \left( \mu _{n}\left( t\right) -z_{n}\right) ^{\top }$
we have
\begin{equation*}
\begin{aligned} \widetilde{\sigma}_n \left( t \right) &= \int_{\tau_n}^t
\mathcal {L}_n \left( s , \sigma_n \left( s \right) \right) ds - z_n
\left( \int_{\tau_n}^{t} \left( B^0_n \, \mu_n \left( s \right) +
b^0_n \left(s \right) \right) ds \right)^{\top} \quad -
\int_{\tau_n}^{t} \left( B^0_n \, \mu_n \left( s \right) + b^0_n
\left(s \right) \right) ds \, z_n^{\top} \\ & \quad -
\int_{\tau_n}^{t} \left( B^0_n \, \mu_n \left( s \right) + b^0_n
\left(s \right) \right) ds \left( \int_{\tau_n}^{t} \left( B^0_n \,
\mu_n \left( s \right) + b^0_n \left(s \right) \right) ds
\right)^{\top} , \end{aligned}
\end{equation*}%
and so (\ref{eq:2.2}), (\ref{eq:2.1}) and (\ref{eq:2.3}) yields
\begin{equation}
\left\vert \widetilde{\sigma }_{n}\left( t\right) \right\vert \leq
K\left( T\right) \left( 1+\left\vert z_{n}\right\vert ^{2}\right)
\left( t-\tau
_{n}\right) \quad \quad \quad \forall t\in \left[ \tau _{n},\tau _{n+1}%
\right] .  \label{eq:2.13}
\end{equation}

Iterating (\ref{eq:1.6}) we obtain
\begin{equation}
z_{n+1}=z_{0}+\int_{t_{0}}^{\tau _{n+1}}\left( B_{n\left( s\right)
}^{0}\,\mu _{n\left( s\right) }\left( s\right) +b_{n\left( s\right)
}^{0}\left( s\right) \right) ds+S_{n+1},  \label{eq:2.8}
\end{equation}%
where $n=0,\ldots ,N-1$, $n\left( t\right) =\max \left\{ n=0,\ldots
,N:\tau _{n}\leq t\right\} $ and
\begin{equation*}
S_{n+1}=\sum_{k=0}^{n}\sqrt{\sigma _{k}\left( \tau _{k+1}\right)
-\mu _{k}\left( \tau _{k+1}\right) \mu _{k}^{\intercal }\left( \tau
_{k+1}\right) }\,\eta _{k}.
\end{equation*}%
Applying H\"{o}lder's inequality we get
\begin{align*}
& \left\vert \int_{t_{0}}^{\tau _{n+1}}\left( B_{n\left( s\right)
}^{0}\,\mu _{n\left( s\right) }\left( s\right) +b_{n\left( s\right)
}^{0}\left(
s\right) \right) ds\right\vert ^{2q} \\
& \leq \left( \tau _{n+1}-t_{0}\right) ^{2q-1}\int_{t_{0}}^{\tau
_{n+1}}\left\vert \left( B_{n\left( s\right) }^{0}\,\mu _{n\left(
s\right) }\left( s\right) +b_{n\left( s\right) }^{0}\left( s\right)
\right) \right\vert ^{2q}ds,
\end{align*}%
and so (\ref{eq:2.17}) and (\ref{eq:2.2}) yield
\begin{equation}
\left\vert \int_{t_{0}}^{\tau _{n+1}}\left( B_{n\left( s\right)
}^{0}\,\mu _{n\left( s\right) }\left( s\right) +b_{n\left( s\right)
}^{0}\left( s\right) \right) ds\right\vert ^{2q}\leq K\left(
T\right) \left( 1+\int_{t_{0}}^{\tau _{n+1}}\left\vert z_{n\left(
s\right) }\right\vert ^{2q}ds\right) .  \label{eq:2.4}
\end{equation}

Set $S_{0}=0$. For any $n=0,\ldots ,N-1$,
\begin{align*}
\mathbb{E}\left( \left\vert S_{n+1}\right\vert ^{2}\right) & =\mathbb{E(}%
S_{n+1}^{\intercal }S_{n+1}) \\
& =\sum_{k=0}^{n}\mathbb{E}\left( \eta _{k}^{\intercal }\left(
\sigma _{k}\left( \tau _{k+1}\right) -\mu _{k}\left( \tau
_{k+1}\right) \mu
_{k}^{\intercal }\left( \tau _{k+1}\right) \right) \eta _{k}\right) \\
& =\sum_{k=0}^{n}\sum_{\ell =1}^{d}\mathbb{E}\left( \sigma
_{k}\left( \tau _{k+1}\right) ^{\ell ,\ell }+\left( \mu _{k}\left(
\tau _{k+1}\right) ^{\ell }\right) ^{2}\right) .
\end{align*}%
Since $\sigma _{k}\left( \tau _{k+1}\right) =\mathbb{E}\left(
Y_{\tau
_{k+1}}Y_{\tau _{k+1}}^{\top }\diagup \mathfrak{F}_{\tau _{k}}\right) $, (%
\ref{eq:2.2}) yields
\begin{equation*}
\mathbb{E}\left( \left\vert S_{n+1}\right\vert ^{2}\right) \leq
\sum_{k=0}^{n}\left( \mathbb{E}\left( \left\vert Y_{\tau
_{k+1}}\right\vert ^{2}\right) +\mathbb{E}\left\vert \mu _{k}\left(
\tau _{k+1}\right) \right\vert ^{2}\right) <+\infty ,
\end{equation*}%
and so $\left( S_{n}\right) _{n=0,\ldots ,N}$ is a $\left( \mathfrak{F}%
_{\tau _{n}}\right) _{n=0,\ldots ,N}$-square integrable martingale.
According to the Burkholder-Davis-Gundy inequality we have
\begin{equation*}
\mathbb{E}\left( \max_{k=0,\ldots ,n}\left( S_{k}^{j}\right)
^{2q}\right)
\leq C_{q}\mathbb{E}\left( \left[ S^{j},S^{j}\right] _{n}^{q}\right) =C_{q}%
\mathbb{E}\left( \sum_{k=0}^{n}\left( \left( \sqrt{\widetilde{\sigma }%
_{k}\left( \tau _{k+1}\right) }\,\eta _{k}\right) ^{j}\right)
^{2}\right) ^{q},
\end{equation*}%
where $C_{q}>0$ and $y^{j}$ stands for the $j$-th coordinate of the vector $%
y $. Applying H\"{o}lder's inequality yields
\begin{align*}
& \left( \sum_{k=0}^{n}\left( \tau _{k+1}-\tau _{k}\right)
^{1/p}\left( \tau
_{k+1}-\tau _{k}\right) ^{1/q}\left( \left( \sqrt{\widetilde{\sigma }%
_{k}\left( \tau _{k+1}\right) }\,\eta _{k}\right) ^{j}\right)
^{2}/\left(
\tau _{k+1}-\tau _{k}\right) \right) ^{q} \\
& \leq \left( \sum_{k=0}^{n}\left( \tau _{k+1}-\tau _{k}\right)
\right)
^{q-1}\sum_{k=0}^{n}\left( \tau _{k+1}-\tau _{k}\right) \left( \left( \sqrt{%
\widetilde{\sigma }_{k}\left( \tau _{k+1}\right) }\,\eta _{k}\right)
^{j}\right) ^{2q}/\left( \tau _{k+1}-\tau _{k}\right) ^{q}
\end{align*}%
with $1/p+1/q=1$. Using $\left\vert \sqrt{\widetilde{\sigma
}_{k}\left( \tau _{k+1}\right) }\right\vert ^{2}=\left\vert
\widetilde{\sigma }_{k}\left( \tau _{k+1}\right) \right\vert $ we
obtain
\begin{align*}
\mathbb{E}\left( \max_{k=0,\ldots ,n}\left\vert S_{k}\right\vert
^{2q}\right) & \leq \left( Td\right)
^{q-1}C_{q}\sum_{k=0}^{n}\mathbb{E} \left( \left( \tau _{k+1}-\tau
_{k}\right) \frac{\left\vert \sqrt{\widetilde{ \sigma }_{k}\left(
\tau _{k+1}\right) }\right\vert ^{2q}}{\left( \tau
_{k+1}-\tau _{k}\right) ^{q}}\left\vert \eta _{k}\right\vert ^{2q}\right) \\
& \leq \left( Td\right) ^{q-1}C_{q}\sum_{k=0}^{n}\mathbb{E}\left(
\left( \tau _{k+1}-\tau _{k}\right) \left( \frac{\left\vert
\widetilde{\sigma } _{k}\left( \tau _{k+1}\right) \right\vert }{\tau
_{k+1}-\tau _{k}}\right) ^{q}\left\vert \eta _{k}\right\vert
^{2q}\right) .
\end{align*}%
Hence (\ref{eq:2.13}) yields
\begin{equation}
\mathbb{E}\left( \max_{k=0,\ldots ,n}\left\vert S_{k}\right\vert
^{2q}\right) \leq K\left( T\right) \mathbb{E}\left( \left\vert \eta
_{0}\right\vert ^{2q}\right) \left( 1+\sum_{k=0}^{n}\left( \tau
_{k+1}-\tau _{k}\right) \mathbb{E(}\left\vert z_{k}\right\vert
^{2q})\right) . \label{eq:2.5}
\end{equation}

Using (\ref{eq:2.8}), (\ref{eq:2.4}) and (\ref{eq:2.5}), together with H\"{o}%
lder's inequality, we get
\begin{equation*}
\mathbb{E}\left( \max_{j=0,\ldots ,n+1}\left\vert z_{j}\right\vert
^{2q}\right) \leq K\left( T\right) \left( \mathbb{E}\left\vert
z_{0}\right\vert ^{2q}+1+\sum_{k=0}^{n}\left( \tau _{k+1}-\tau
_{k}\right) \mathbb{E(}\left\vert z_{k}\right\vert ^{2q})\right) .
\end{equation*}%
The discrete time Gronwall-Bellman lemma now leads to
(\ref{eq:2.6}).

We proceed to show (\ref{eq:2.7}). Using H\"{o}lder's inequality and (\ref%
{eq:2.2}) we obtain
\begin{eqnarray*}
\left\vert \int_{\tau _{n}}^{\tau _{n+1}}\left( B_{n}^{0}\,\mu
_{n}\left( s\right) +b_{n}^{0}\left( s\right) \right) ds\right\vert
^{2q} &\leq &\left( \tau _{n+1}-\tau _{n}\right) ^{2q-1}\int_{\tau
_{n}}^{\tau _{n+1}}\left( \left\vert B_{n}^{0}\right\vert \left\vert
\mu _{n}\left( s\right) \right\vert +\left\vert b_{n}^{0}\left(
s\right) \right\vert \right) ^{2q}ds
\\
&\leq &K\left( T\right) \left( \tau _{n+1}-\tau _{n}\right)
^{2q}\left( 1+\left\vert z_{n}\right\vert ^{2q}\right) .
\end{eqnarray*}
By (\ref{eq:2.13}),
\begin{align*}
\left\vert \sqrt{\widetilde{\sigma }_{n}\left( \tau _{n+1}\right)
}\,\eta _{n}\right\vert ^{2q}& \leq \left\vert
\sqrt{\widetilde{\sigma }_{n}\left( \tau _{n+1}\right) }\right\vert
^{2q}\left\vert \eta _{n}\right\vert ^{2q} \\ & =\left\vert
\widetilde{\sigma }_{n}\left( \tau _{n+1}\right)
\right\vert ^{q}\left\vert \eta _{n}\right\vert ^{2q} \\
& \leq K\left( T\right) \left( 1+\left\vert z_{n}\right\vert
^{2q}\right) \left( \tau _{n+1}-\tau _{n}\right) ^{q}\left\vert \eta
_{n}\right\vert ^{2q}.
\end{align*}%
Hence
\begin{equation*}
\mathbb{E}\left( \left\vert \sqrt{\widetilde{\sigma }_{n}\left( \tau
_{n+1}\right) }\,\eta _{n}\right\vert ^{2q}\diagup
\mathfrak{F}_{\tau _{n}}\right) \leq K\left( T\right) \left(
1+\left\vert z_{n}\right\vert ^{2q}\right) \left( \tau _{n+1}-\tau
_{n}\right) ^{q}\mathbb{E}\left( \left\vert \eta _{n}\right\vert
^{2q}\right) .
\end{equation*}%
This implies (\ref{eq:2.7}), because
\begin{align*}
\mathbb{E}\left( \left\vert z_{n+1}-z_{n}\right\vert ^{2q}\diagup \mathfrak{F%
}_{\tau _{n}}\right) & \leq 2^{2q-1}\mathbb{E}\left( \left\vert
\int_{\tau _{n}}^{\tau _{n+1}}\left( B_{n}^{0}\,\mu _{n}\left(
s\right)
+b_{n}^{0}\left( s\right) \right) ds\right\vert ^{2q}\diagup \mathfrak{F}%
_{\tau _{k}}\right) \\
& \quad +2^{2q-1}\mathbb{E}\left( \left\vert \sqrt{\widetilde{\sigma }%
_{n}\left( \tau _{n+1}\right) }\,\eta _{n}\right\vert ^{2q}\diagup \mathfrak{%
F}_{\tau _{n}}\right) .
\end{align*}
\end{proof}

\begin{lemma}
\label{lem:CompMomentos} Assume the hypothesis of Theorem \ref{th:convScheme}%
. Let
\begin{equation*}
\chi _{n+1}=f\left( \tau _{n},z_{n}\right) \left( \tau _{n+1}-\tau
_{n}\right) +\sum_{k=1}^{m}g^{k}\left( \tau _{n},z_{n}\right) \left(
W_{\tau _{n+1}}^{k}-W_{\tau _{k}}^{k}\right) .
\end{equation*}%
Then, for all $n=0,\ldots ,N-1$, it is obtained that%
\begin{equation}
\left\vert \mathbb{E}\left( \left( z_{n+1}-z_{n}\right) \diagup \mathfrak{F}%
_{\tau _{n}}\right) -\mathbb{E}\left( \chi _{n+1}\diagup
\mathfrak{F}_{\tau _{n}}\right) \right\vert \leq K\left( T\right)
\left( \tau _{n+1}-\tau _{n}\right) ^{2}\left( 1+\left\vert
z_{n}\right\vert \right) , \label{eq:2.9}
\end{equation}%
\begin{equation}
\left\vert \mathbb{E}\left( \left( z_{n+1}-z_{n}\right) \left(
z_{n+1}-z_{n}\right) ^{\top }\diagup \mathfrak{F}_{\tau _{n}}\right) -%
\mathbb{E}\left( \chi _{n+1}\chi _{n+1}^{\top }\diagup
\mathfrak{F}_{\tau _{n}}\right) \right\vert \leq K\left( T\right)
\left( \tau _{n+1}-\tau _{n}\right) ^{2}\left( 1+\left\vert
z_{n}\right\vert ^{2}\right) , \label{eq:2.10}
\end{equation}
and%
\begin{equation}
\begin{array}{c}
\left\vert \mathbb{E}\left( \left( z_{n+1}-z_{n}\right) ^{\ell
}\left(
z_{n+1}-z_{n}\right) \left( z_{n+1}-z_{n}\right) ^{\top }\diagup \mathfrak{F}%
_{\tau _{n}}\right) -\mathbb{E}\left( \chi _{n+1}^{\ell }\chi
_{n+1}\chi
_{n+1}^{\top }\diagup \mathfrak{F}_{\tau _{n}}\right) \right\vert \\
\leq K\left( T\right) \left( \tau _{n+1}-\tau _{n}\right) ^{2}\left(
1+\left\vert z_{n}\right\vert ^{2}\right).%
\end{array}
\label{eq:2.11}
\end{equation}
\end{lemma}

\begin{proof}
Since $B_{n}^{0}\,z_{n}+b_{n}^{0}\left( \tau _{n}\right) =f\left(
\tau _{n},z_{n}\right) $,
\begin{eqnarray*}
\mu _{n}\left( \tau _{n+1}\right) -z_{n}-f\left( \tau
_{n},z_{n}\right) \left( \tau _{n+1}-\tau _{n}\right)  &=&\int_{\tau
_{n}}^{\tau _{n+1}}\left( B_{n}^{0}\,\mu _{n}\left( s\right)
+b_{n}^{0}\left( s\right) -f\left( \tau
_{n},z_{n}\right) \right) ds \\
&=&\int_{\tau _{n}}^{\tau _{n+1}}\left( B_{n}^{0}\left( \mu
_{n}\left( s\right) -z_{n}\right) +b_{n}^{0}\left( s\right)
-b_{n}^{0}\left( \tau _{n}\right) \right) ds.
\end{eqnarray*}
Using (\ref{eq:1.5}) and (\ref{eq:2.2}) we deduce that
\begin{eqnarray}
\left\vert \mu _{n}\left( \tau _{n+1}\right) -z_{n}-f\left( \tau
_{n},z_{n}\right) \left( \tau _{n+1}-\tau _{n}\right) \right\vert
&\leq &K\int_{\tau _{n}}^{\tau _{n+1}}\left( \left\vert \mu
_{n}\left( s\right)
-z_{n}\right\vert +s-\tau _{n}\right) ds  \notag \\
&\leq &K\int_{\tau _{n}}^{\tau _{n+1}}\int_{\tau _{n}}^{s}\left\vert
B_{n}^{0}\,\mu _{n}\left( r\right) +b_{n}^{0}\left( s\right)
\right\vert
drds+K\left( \tau _{n+1}-\tau _{n}\right) ^{2}  \notag \\
&\leq &K\left( T\right) \left( \tau _{n+1}-\tau _{n}\right)
^{2}\left( 1+\left\vert z_{n}\right\vert \right) .  \label{eq:2.14}
\end{eqnarray}
Since
\begin{equation*}
\left\vert \mathbb{E}\left( \left( z_{n+1}-z_{n}\right) \diagup \mathfrak{F}%
_{\tau _{n}}\right) -\mathbb{E}\left( \chi _{n+1}\diagup
\mathfrak{F}_{\tau _{n}}\right) \right\vert =\left\vert \mu
_{n}\left( \tau _{n+1}\right) -z_{n}-f\left( \tau _{n},z_{n}\right)
\left( \tau _{n+1}-\tau _{n}\right) \right\vert ,
\end{equation*}%
(\ref{eq:2.14}) yields (\ref{eq:2.9}).

From
\begin{equation*}
\mathbb{E}\left( \chi _{n+1}\chi _{n+1}^{\top }\diagup
\mathfrak{F}_{\tau _{n}}\right) =f\left( \tau _{n},z_{n}\right)
f\left( \tau _{n},z_{n}\right) ^{\top }\left( \tau _{n+1}-\tau
_{n}\right) ^{2}+\sum_{k=1}^{m}g^{k}\left( \tau _{n},z_{n}\right)
g^{k}\left( \tau _{n},z_{n}\right) ^{\top }\left( \tau _{n+1}-\tau
_{n}\right)
\end{equation*}%
we obtain%
\begin{equation}
\left\vert \mathbb{E}\left( \chi _{n+1}\chi _{n+1}^{\top }\diagup \mathfrak{F%
}_{\tau _{n}}\right) -\sum_{k=1}^{m}g^{k}\left( \tau
_{n},z_{n}\right) (g^{k}\left( \tau _{n},z_{n}\right) )^{\top
}\left( \tau _{n+1}-\tau _{n}\right) \right\vert \leq K\left(
T\right) \left( \tau _{n+1}-\tau _{n}\right) ^{2}\left( 1+\left\vert
z_{n}\right\vert ^{2}\right) . \label{eq:2.12}
\end{equation}

As in the proof of Lemma \ref{lem:CotaCrecimiento}, we define $\widetilde{%
\sigma }_{n}\left( t\right) :=\sigma _{n}\left( t\right) -\mu
_{n}\left( t\right) \mu _{n}\left( t\right) ^{\top }$ for any $t\in
\left[ \tau _{n},\tau _{n+1}\right] $. Then
\begin{align*}
\mathbb{E}\left( \left( z_{n+1}-z_{n}\right) \left(
z_{n+1}-z_{n}\right) ^{\top }\diagup \mathfrak{F}_{\tau _{n}}\right)
& =\left( \mu _{n}\left( \tau _{n+1}\right) -z_{n}\right) \left( \mu
_{n}\left( \tau _{n+1}\right) -z_{n}\right) ^{\top }
+\widetilde{\sigma }_{n}\left( \tau _{n+1}\right) .
\end{align*}%
Since
\begin{align*}
\widetilde{\sigma }_{n}\left( \tau _{n+1}\right) & =\sigma
_{n}\left( \tau _{n+1}\right) -z_{n}z_{n}^{\top }-z_{n}\left( \mu
_{n}\left( \tau _{n+1}\right) -z_{n}\right) ^{\top }-\left( \mu
_{n}\left( \tau
_{n+1}\right) -z_{n}\right) z_{n}^{\top } \\
& \quad -\left( \mu _{n}\left( \tau _{n+1}\right) -z_{n}\right)
\left( \mu _{n}\left( \tau _{n+1}\right) -z_{n}\right) ^{\top },
\end{align*}%
applying (\ref{eq:2.14}) yields
\begin{align*}
& \left\vert \mathbb{E}\left( \left( z_{n+1}-z_{n}\right) \left(
z_{n+1}-z_{n}\right) ^{\top }\diagup \mathfrak{F}_{\tau _{n}}\right)
-\sigma
_{n}\left( \tau _{n+1}\right) +z_{n}z_{n}^{\top }\right. \\
& \hspace{3cm}\left. +z_{n}f\left( \tau _{n},z_{n}\right) ^{\top
}\left( \tau _{n+1}-\tau _{n}\right) +f\left( \tau _{n},z_{n}\right)
z_{n}^{\top
}\left( \tau _{n+1}-\tau _{n}\right) \right\vert \\
& \leq K\left( T\right) \left( \tau _{n+1}-\tau _{n}\right)
^{2}\left( 1+\left\vert z_{n}\right\vert ^{2}\right) .
\end{align*}%
Using (\ref{eq:2.17}), (\ref{eq:2.2}) and (\ref{eq:2.1}), together
with Hypothesis \ref{hyp:Hip1}, we deduce that
\begin{equation*}
\left\vert \mathcal{L}_{n}\left( s,\sigma _{n}\left( s\right) \right) -%
\mathcal{L}_{n}\left( \tau _{n},z_{n}z_{n}^{\top }\right)
\right\vert \leq K\left( T\right) \left( s-\tau _{n}\right) \left(
1+\left\vert z_{n}\right\vert ^{2}\right) ,
\end{equation*}%
and so
\begin{eqnarray*}
\left\vert \sigma _{n}\left( \tau _{n+1}\right) -z_{n}z_{n}^{\top }-\mathcal{%
L}_{n}\left( \tau _{n},z_{n}z_{n}^{\top }\right) \left( \tau
_{n+1}-\tau _{n}\right) \right\vert  &\leq &\int_{\tau _{n}}^{\tau
_{n+1}}\left\vert
\mathcal{L}_{n}\left( s,\sigma _{n}\left( s\right) \right) -\mathcal{L}%
_{n}\left( \tau _{n},z_{n}z_{n}^{\top }\right) \right\vert ds \\
&\leq &K\left( T\right) \left( \tau _{n+1}-\tau _{n}\right)
^{2}\left( 1+\left\vert z_{n}\right\vert ^{2}\right) .
\end{eqnarray*}
Therefore
\begin{align}
& \left\vert \mathbb{E}\left( \left( z_{n+1}-z_{n}\right) \left(
z_{n+1}-z_{n}\right) ^{\top }\diagup \mathfrak{F}_{\tau _{n}}\right)
-\sum_{k=1}^{m}g^{k}\left( \tau _{n},z_{n}\right) (g^{k}\left( \tau
_{n},z_{n}\right) )^{\top }\left( \tau _{n+1}-\tau _{n}\right)
\right\vert
\label{eq:2.15} \\
& \leq K\left( T\right) \left( \tau _{n+1}-\tau _{n}\right)
^{2}\left( 1+\left\vert z_{n}\right\vert ^{2}\right) ,  \notag
\end{align}%
because
\begin{equation*}
\mathcal{L}_{n}\left( \tau _{n},z_{n}z_{n}^{\top }\right)
=z_{n}f\left( \tau _{n},z_{n}\right) ^{\top }+f\left( \tau
_{n},z_{n}\right) z_{n}^{\top }+\sum_{k=1}^{m}g^{k}\left( \tau
_{n},z_{n}\right) g^{k}\left( \tau _{n},z_{n}\right) ^{\top }.
\end{equation*}%
Combining (\ref{eq:2.12}) with (\ref{eq:2.15}) we get
(\ref{eq:2.10}).

A careful computation shows
\begin{eqnarray*}
\mathbb{E}\left( \chi _{n+1}^{\ell }\chi _{n+1}\chi _{n+1}^{\top
}\diagup \mathfrak{F}_{\tau _{n}}\right)  &=&f\left( \tau
_{n},z_{n}\right) ^{\ell }f\left( \tau _{n},z_{n}\right) f\left(
\tau _{n},z_{n}\right) ^{\top }\left( \tau _{n+1}-\tau _{n}\right)
^{3}+f\left( \tau _{n},z_{n}\right)
^{\ell }G_{n}G_{n}^{\top }\left( \tau _{n+1}-\tau _{n}\right) ^{2} \\
&&+f\left( \tau _{n},z_{n}\right) \left( G_{n}G_{n}^{\top }\right)
^{\ell ,\cdot }\left( \tau _{n+1}-\tau _{n}\right) ^{2}+\left(
G_{n}G_{n}^{\top }\right) ^{\cdot ,\ell }f\left( \tau
_{n},z_{n}\right) ^{\top }\left( \tau _{n+1}-\tau _{n}\right) ^{2},
\end{eqnarray*}%
where $G_{n}$ is the $\mathbb{R}^{d\times m}$-matrix whose $\left(
i,j\right) $-th element is the $i$-th entry of $g^{j}\left( \tau
_{n},z_{n}\right) $. Similarly,%
\begin{eqnarray*}
\mathbb{E}\left( \left( z_{n+1}-z_{n}\right) ^{\ell }\left(
z_{n+1}-z_{n}\right) \left( z_{n+1}-z_{n}\right) ^{\top }\diagup \mathfrak{F}%
_{\tau _{n}}\right)  &=&\left( \mu _{n}\left( \tau _{n+1}\right)
-z_{n}\right) ^{\ell }\widetilde{\sigma }_{n}\left( \tau _{n+1}\right)  \\
&&+\left( \mu _{n}\left( \tau _{n+1}\right) -z_{n}\right) \widetilde{\sigma }%
_{n}\left( \tau _{n+1}\right) ^{\ell ,\cdot } \\&&
+\widetilde{\sigma }_{n}\left( \tau _{n+1}\right) ^{\cdot ,\ell
}\left( \mu _{n}\left( \tau _{n+1}\right)
-z_{n}\right) ^{\top } \\
&&+\left( \mu _{n}\left( \tau _{n+1}\right) -z_{n}\right) ^{\ell
}\left( \mu _{n}\left( \tau _{n+1}\right) -z_{n}\right) \left( \mu
_{n}\left( \tau _{n+1}\right) -z_{n}\right) ^{\top }.
\end{eqnarray*}
The last two inequalities imply (\ref{eq:2.11}), which completes the
proof.
\end{proof}

\section{Numerical Simulations}

In this section, numerical simulations are presented in order to
illustrate the performance of Scheme \ref{scheme:MM1}. This involves
the numerical calculation of known expresions for functionals of two
SDEs: a bilinear equation with random oscillatory dynamics, and a
renowned nonlinear test equation. Pad\'{e} method with scaling and
squaring strategy (see, e.g., \cite{Moler03}) was used to compute
the exponential matrix in (\ref{eq:1.7}) and (\ref{eq:1.8}), whereas
the squared root of the matrix $\sigma _{n}\left( \tau _{n+1}\right)
-\mu _{n}\left( \tau _{n+1}\right) \mu _{n}^{\intercal }\left( \tau
_{n+1}\right) $ in (\ref{eq:1.6}) was computed by means of the
singular value decomposition (see, e.g., \cite{Golub 1996}). $\eta
_{n}^{k}$ in (\ref{eq:1.6}) was set as a two-point distributed
random variable with probability $P(\eta _{n}^{k}=\pm 1)=1/2$ for all $n=0,..,N-1$ and $%
k=1,..,m$. All simulations were carried out in Matlab2014a.

\begin{example}
\textbf{Bilinear SDE with random oscillatory dynamics.}
\begin{equation}
dX_{t}=\alpha \left[
\begin{array}{cc}
0 & 1 \\
-1 & 0%
\end{array}%
\right] X_{t}dt+\rho _{1}\left[
\begin{array}{cc}
0 & 1 \\
-1 & 0%
\end{array}%
\right] X_{t}dW_{t}^{1}+\rho _{2}\left[
\begin{array}{cc}
1 & 0 \\
0 & 1%
\end{array}%
\right] X_{t}dW_{t}^{2},  \label{Eq Ej1}
\end{equation}%
for all $t\in \lbrack 0,12.5625]$, initial condition $%
(X_{0}^{1},X_{0}^{2})=(1,2)$, and parameters $\alpha =10$, $\rho
_{1}=0.1$ and $\rho _{2}=2\rho _{1}$.
\end{example}

Since $\left[
\begin{array}{cc}
1 & 0 \\
0 & 1%
\end{array}%
\right] $ commutates with $\left[
\begin{array}{cc}
0 & 1 \\
-1 & 0%
\end{array}%
\right] $, the solution of (\ref{Eq Ej1}) is given by
\begin{equation}
X_{t}=\exp \left( \left[
\begin{array}{cc}
(\rho _{1}^{2}-\rho _{2}^{2})/2 & \alpha \\
-\alpha & (\rho _{1}^{2}-\rho _{2}^{2})/2%
\end{array}%
\right] t+\left[
\begin{array}{cc}
0 & \rho _{1} \\
-\rho _{1} & 0%
\end{array}%
\right] W_{t}^{1}+\left[
\begin{array}{cc}
\rho _{2} & 0 \\
0 & \rho _{2}%
\end{array}%
\right] W_{t}^{2}\right)  \label{Exact Sol Ej1}
\end{equation}%
(see, e.g., \cite{Arnold 1974}, p. 144). From Theorem 3 in
\cite{Jimenez 2012}, the mean $m_{t}$ and variance $v_{t}$ of
$X_{t}$ are given by the expresions
\begin{equation}
m_{t}=X_{0}+L_{2}\exp (Ht)u_{0}  \label{Ej1 media}
\end{equation}%
and%
\begin{equation}
vec(v_{t})=L_{1}\exp (Ht)u_{0}-vec(m_{t}m_{t}^{\intercal }),
\label{Ej1 Var}
\end{equation}%
where the matrices $L_{1}$, $L_{2}$, $H$ and the vector $u_{0}$ are
defined as
\begin{equation*}
H=\left[
\begin{array}{ccc}
A & 0 & 0 \\
0 & 0 & 0 \\
0 & 0 & C%
\end{array}%
\right] \in
\mathbb{R}
^{8\times 8},\text{ \ \ \ \ \ }u_{0}=\left[
\begin{array}{c}
vec(X_{0}X_{0}^{\intercal }) \\
1 \\
r%
\end{array}%
\right] \in
\mathbb{R}
^{8},
\end{equation*}%
\begin{equation*}
L_{1}=[%
\begin{array}{cc}
I_{4} & 0_{4}%
\end{array}%
]\in
\mathbb{R}
^{4\times 8}\text{ \ \ \ and \ \ \ \ \ }L_{2}=[%
\begin{array}{ccc}
0_{2\times 5} & I_{2} & 0_{2\times 1}%
\end{array}%
]\text{\ }\in
\mathbb{R}
^{2\times 8}
\end{equation*}%
with%
\begin{equation*}
A=\left[
\begin{array}{cccc}
\rho _{2}^{2} & \alpha & \alpha & \rho _{1}^{2} \\
-\alpha & \rho _{2}^{2} & -\rho _{1}^{2} & \alpha \\
-\alpha & -\rho _{1}^{2} & \rho _{2}^{2} & \alpha \\
\rho _{1}^{2} & -\alpha & -\alpha & \rho _{2}^{2}%
\end{array}%
\right] \in
\mathbb{R}
^{4\times 4},\text{ \ \ \ }C=\left[
\begin{array}{ccc}
0 & \alpha & \alpha X_{0}^{2} \\
-\alpha & 0 & -\alpha X_{0}^{1} \\
0 & 0 & 0%
\end{array}%
\right] \in
\mathbb{R}
^{3\times 3}\text{ \ \ \ and \ \ \ \ \ }r=\left[
\begin{array}{c}
0 \\
0 \\
1%
\end{array}%
\right] \in
\mathbb{R}
^{3}.
\end{equation*}

First, we compare the  exact values (\ref{Ej1 media})-(\ref{Ej1
Var}) for the mean and variance of $X_t$ with their estimates
obtained via Monte Carlo simulations. For this purpose, $M$
realizations $X_{\tau _{n}}^{\{i\}}$ of the exact solution and
$z_{n}^{\{i\}}$ of the Scheme \ref{scheme:MM1} were computed on an
uniform time partition $\tau _{n}=n\Delta$, with $\Delta =1/2^{6}$,
$n=0,..,N$, and $N=804$. Then, with the estimates
\begin{equation*}
\overline{m}_{\tau _{n}}=\frac{1}{M}\sum\limits_{i=1}^{M}X_{\tau
_{n}}^{\{i\}}\text{ \ \ \ \ \ \ \ and \ \ \ \ \ \ \ \ \
}\widehat{m}_{\tau
_{n}}=\frac{1}{M}\sum\limits_{i=1}^{M}z_{n}^{\{i\}}
\end{equation*}%
for the mean, and
\begin{equation*}
\overline{v}_{\tau _{n}}=\frac{1}{M}\sum\limits_{i=1}^{M}X_{\tau
_{n}}^{\{i\}}(X_{\tau _{n}}^{\{i\}})^{\intercal }-\overline{m}_{\tau _{n}}%
\overline{m}_{\tau _{n}}^{\intercal }\text{ \ \ \ \ \ \ \ and \ \ \
\ \ \ \
\ \ }\widehat{v}_{\tau _{n}}=\frac{1}{M}\sum\limits_{i=1}^{M}z_{n}^{\{i%
\}}(z_{n}^{\{i\}})^{\intercal }-\widehat{m}_{\tau
_{n}}\widehat{m}_{\tau _{n}}^{\intercal }
\end{equation*}%
for the variance, the errors
\begin{equation*}
\overline{e}_{\tau _{n}}^{[1]}=\left\vert m_{\tau _{n}}^{1}-\overline{m}%
_{\tau _{n}}^{1}\right\vert \text{ \ \ \ \ and \ \ \ \ \
}\widehat{e}_{\tau _{n}}^{[1]}=\left\vert m_{\tau
_{n}}^{1}-\widehat{m}_{\tau _{n}}^{1}\right\vert
\end{equation*}%
\begin{equation*}
\overline{e}_{\tau _{n}}^{[2]}=\left\vert m_{\tau _{n}}^{2}-\overline{m}%
_{\tau _{n}}^{2}\right\vert \text{ \ \ \ \ and \ \ \ \ \
}\widehat{e}_{\tau _{n}}^{[2]}=\left\vert m_{\tau
_{n}}^{2}-\widehat{m}_{\tau _{n}}^{2}\right\vert
\end{equation*}%
\begin{equation*}
\overline{e}_{\tau _{n}}^{[3]}=\left\vert v_{\tau _{n}}^{1,1}-\overline{v}%
_{\tau _{n}}^{1,1}\right\vert \text{ \ \ \ \ and \ \ \ \ \ }\widehat{e}%
_{\tau _{n}}^{[3]}=\left\vert v_{\tau _{n}}^{1,1}-\widehat{v}_{\tau
_{n}}^{1,1}\right\vert
\end{equation*}%
\begin{equation*}
\overline{e}_{\tau _{n}}^{[4]}=\left\vert v_{\tau _{n}}^{2,2}-\overline{v}%
_{\tau _{n}}^{2,2}\right\vert \text{ \ \ \ \ and \ \ \ \ \ }\widehat{e}%
_{\tau _{n}}^{[4]}=\left\vert v_{\tau _{n}}^{2,2}-\widehat{v}_{\tau
_{n}}^{2,2}\right\vert
\end{equation*}%
\begin{equation*}
\overline{e}_{\tau _{n}}^{[5]}=\left\vert v_{\tau _{n}}^{1,2}-\overline{v}%
_{\tau _{n}}^{1,2}\right\vert \text{ \ \ \ \ and \ \ \ \ \ }\widehat{e}%
_{\tau _{n}}^{[5]}=\left\vert v_{\tau _{n}}^{1,2}-\widehat{v}_{\tau
_{n}}^{1,2}\right\vert
\end{equation*}%
were evaluated. Here, for computing $X_{\tau _{n}}^{\{i\}}$, the
realization of the Wiener process $(W_{\tau _{n}}^{1},W_{\tau
_{n}}^{2})$ was simulated
as $W_{\tau _{n}}^{k}=\sum\limits_{j=1}^{n}\Delta W_{\tau _{j}}^{k}$ and $%
\Delta W_{\tau _{j}}^{k}\sim \sqrt{\Delta }\mathcal{N(}0,1)$ for each $k=1,2$%
, where $\mathcal{N(}0,1)$ is a Gaussian random variable with zero
mean and variance 1.

Figure 1 shows the exact values of $m_{\tau _{n}},v_{\tau _{n}}$
versus their approximations $\widehat{m}_{\tau
_{n}},\widehat{v}_{\tau _{n}}$ obtained from $M=2^{16}$ simulations
of Scheme \ref{scheme:MM1}. Observe that there is not visual
difference among these values. Table \ref{TableI} presents the
errors $\widehat{e}^{[l]}=\underset{n}{\max }\{\widehat{e} _{\tau
_{n}}^{[l]}\}$ and \ $\overline{e}^{[l]}=\underset{n}{\max }\{
\overline{e}_{\tau _{n}}^{[l]}\}$ of the estimated value of the mean
and variance of (\ref{Eq Ej1}) computed with different number of
simulations $M$. As it was expected, these errors decrease as the
number of simulations $M$ increases. It is well known that the error
$e$ of the sampling mean of the Monte Carlo method decrease with the
inverse of the square root of the number of simulations
\cite{Kloeden 1995}, i.e.,
\begin{equation*}
e\propto \frac{1}{M^{\gamma }}
\end{equation*}%
with $\gamma =0.5$. A roughly estimator $\gamma _{\tau _{n}}^{[l]}$ of $%
\gamma $ for the errors $\widehat{e}_{\tau _{n}}^{[l]}$ and $\overline{e}%
_{\tau _{n}}^{[l]}$ was computed as minus the slope of the straight
line fitted to the set of six points $\left\{ (\log
_{2}(M_{k}),\,\log _{2}(e_{\tau
_{n}}^{[l]}(M_{k}))):\;M_{k}=2^{k},k=8,10,12,14,16,18\right\} $.
Table \ref{TableII} shows the average \ \ \ \ \ \ \ \ \
\begin{equation*}
\widetilde{\gamma }^{[l]}=\frac{1}{N}\sum\limits_{n=1}^{N}\gamma
_{\tau _{n}}^{[l]}
\end{equation*}%
for each type of error and its corresponding standard deviation
\begin{equation*}
s^{[l]}=\sqrt{\frac{1}{N-1}\sum\limits_{n=1}^{N}(\gamma _{\tau
_{n}}^{[l]}-\gamma ^{\lbrack l]})^{2}}.
\end{equation*}

\begin{figure}[h]
\centering
\includegraphics[width=6.5in]{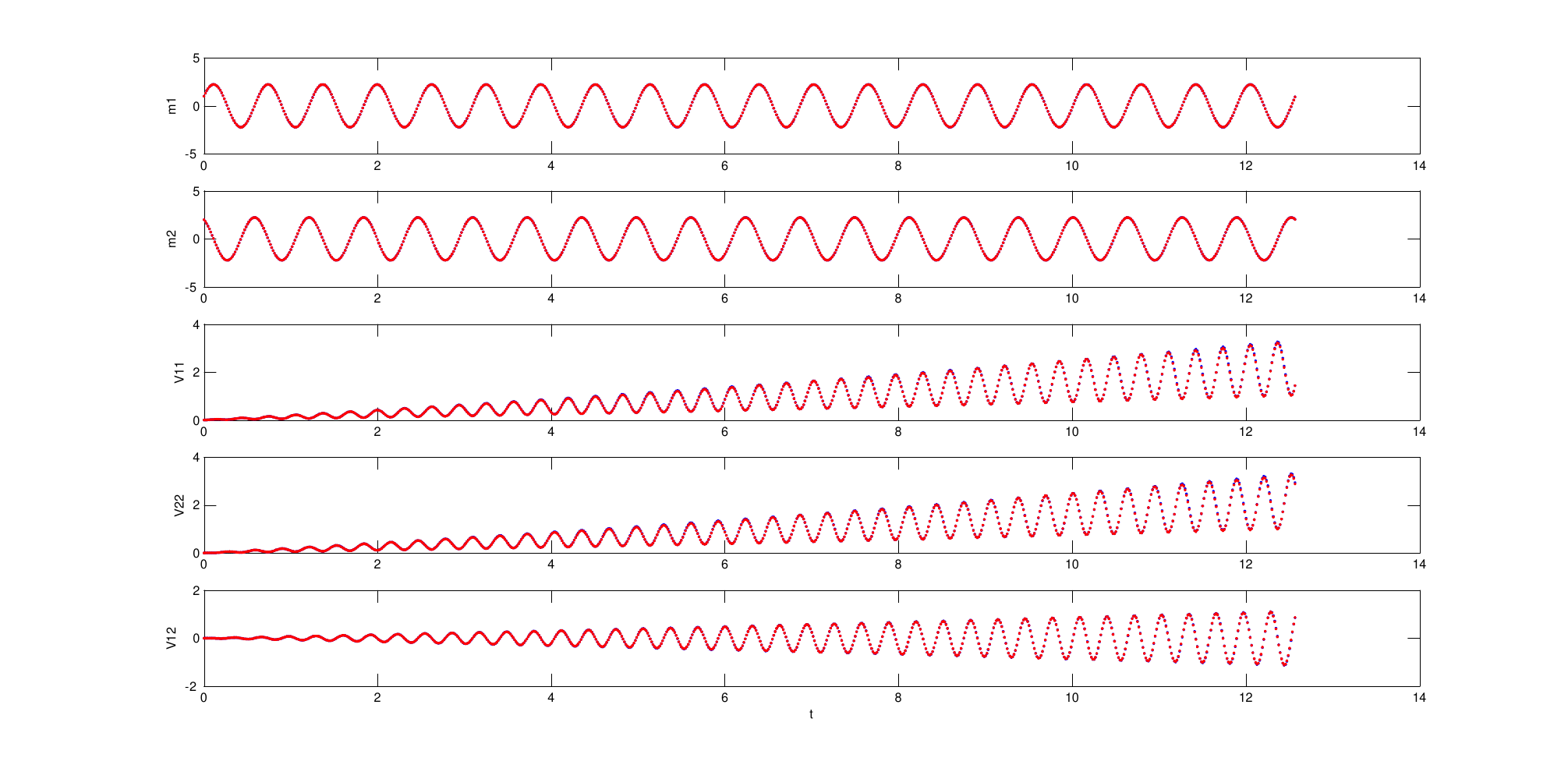}
\caption{Integration of Example 1. Exact values of
$m_{t_{n}},v_{t_{n}}$ and their approximations
$\widehat{m}_{t_{n}},\widehat{v}_{t_{n}}$ computed via
Monte Carlos with $M=2^{16}$ realizations of the Scheme \protect\ref%
{scheme:MM1}.} \label{Fig1}
\end{figure}

Results of Tables \ref{TableI} and \ref{TableII}, together with
Figure 1, indicate that the estimators for the mean and variance of
(\ref{Eq Ej1}) obtained by means of the simulations of the exact
solution (\ref{Exact Sol Ej1}) and Scheme \ref{scheme:MM1} are quite
similar. This is an expected result since the first two moments of
the linear SDEs and Scheme 1 are "equal" (up to the precision of the
floating-point arithmetic in the numerical computation of the
involved exponential and square root matrices).

\begin{table}
\begin{center}
\begin{tabular}{c||cccccc}
/$M$ & $2^{8}$ & $2^{10}$ & $2^{12}$ & $2^{14}$ & $2^{16}$ & $2^{18}$ \\
\hline\hline $\widehat{e}^{[1]}$ & 0.10710 & 0.05228 & 0.04536 &
0.01508 & 0.00686 &
0.00275 \\
$\widehat{e}^{[2]}$ & 0.10643 & 0.05025 & 0.04469 & 0.01433 &
0.00647 &
0.00304 \\
$\widehat{e}^{[3]}$ & 0.43411 & 0.25916 & 0.18319 & 0.29184 &
0.07244 &
0.03181 \\
$\widehat{e}^{[4]}$ & 0.39102 & 0.29529 & 0.21413 & 0.29496 &
0.07726 &
0.02753 \\
$\widehat{e}^{[5]}$ & 0.23859 & 0.14325 & 0.15463 & 0.16961 &
0.05187 & 0.02450 \\ \hline $\overline{e}^{[1]}$ & 0.27037 & 0.02964
& 0.02101 & 0.02108 & 0.01487 &
0.00376 \\
$\overline{e}^{[2]}$ & 0.27626 & 0.04147 & 0.02327 & 0.02227 &
0.01452 &
0.00347 \\
$\overline{e}^{[3]}$ & 0.92465 & 0.35064 & 0.18339 & 0.15513 &
0.06024 &
0.02482 \\
$\overline{e}^{[4]}$ & 0.89503 & 0.39518 & 0.17646 & 0.14678 &
0.05655 &
0.02346 \\
$\overline{e}^{[5]}$ & 0.36642 & 0.24892 & 0.10664 & 0.08899 &
0.02785 &
0.01101%
\end{tabular}
\end{center}
\caption{{\small Values of the errors }$\widehat{e}^{[l]}${\small \ and }$%
\overline{e}^{[l]}${\small \ versus number of simulations
}$M${\small \ in the Example 1.}} \label{TableI}
\end{table}

\begin{table}
\begin{center}
\begin{tabular}{c||ccccc|ccccc}
& $\widehat{e}^{[1]}$ & $\widehat{e}^{[2]}$ & $\widehat{e}^{[3]}$ & $%
\widehat{e}^{[4]}$ & $\widehat{e}^{[5]}$ & $\overline{e}^{[1]}$ & $\overline{%
e}^{[2]}$ & $\overline{e}^{[3]}$ & $\overline{e}^{[4]}$ & $\overline{e}%
^{[5]} $ \\ \hline\hline $\widetilde{\gamma }$ & 0.52 & 0.53 & 0.44
& 0.44 & 0.41 & 0.44 & 0.44 & 0.44
& 0.43 & 0.45 \\
$std$ & 0.16 & 0.16 & 0.20 & 0.20 & 0.21 & 0.18 & 0.19 & 0.21 & 0.21 & 0.20%
\end{tabular}
\end{center}
\caption{{\small Average }$\widetilde{\gamma }${\small \ and
standard deviation }$std${\small \ of the estimators for the rate of
convergency }$ \gamma =1/2${\small \ of the Monte Carlo simulations
in\ the Example 1.}} \label{TableII}
\end{table}

\begin{table}
\begin{center}
\begin{tabular}{c||cccccc}
/$M$ & $2^{8}$ & $2^{10}$ & $2^{12}$ & $2^{14}$ & $2^{16}$ & $2^{18}$ \\
\hline\hline
$r^{[1]}$ & 0.0522 & 0.0177 & 0.0105 & 0.0037 & 0.0016 & 0.0010 \\
$r^{[2]}$ & 0.0534 & 0.0159 & 0.0106 & 0.0037 & 0.0014 & 0.0010%
\end{tabular}
\end{center}
\caption{{\small Relative error }$r^{[l]}$ {\small in the
computation of the functionals }$\overline{h}_{\tau
_{n}}^{[l]}${\small \ and }$\widehat{h} _{\tau _{n}}^{[l]}${\small \
with different number of simulations }${\small M }${\small \ in the
Example 1.}} \label{TableIII}
\end{table}

In addition, let us compute the relative difference
\begin{equation*}
r^{[l]} \left( M \right) = \max_n \left\{ \left\vert
(\overline{h}_{\tau _{n}}^{[l]}-\widehat{h}_{\tau
_{n}}^{[l]})/\overline{h}_{\tau_{n}}^{[l]}\right\vert \right\}
\end{equation*}
between the approximations
\begin{equation*}
\overline{h}_{\tau
_{n}}^{[l]}=\frac{1}{M}\sum\limits_{i=1}^{M}\arctan
\left(1+ \left( \left( X_{\tau _{n}}^{l} \right)^{\{i\}} \right)^{2} \right)\text{ \ \ \ \ \ \ \ and \ \ \ \ \ \ \ }%
\widehat{h}_{\tau_{n}}^{[l]}=\frac{1}{M}\sum\limits_{i=1}^{M}
\arctan \left(1+\left( \left(z_{n}^{l} \right)^{\{i\}}
\right)^{2}\right)
\end{equation*}
of the nonlinear functionals $h_{\tau _{n}}^{[l]}= \mathbb{E}
\left(\arctan \left(1+(X_{\tau _{n}}^{l})^{2} \right) \right)$,
with  $l=1,2$. Table \ref{TableIII} displays the values of $%
r^{[l]}$ for different values of $M$. As it was also expected, $%
r^{[l]}$ goes to zero as the number of simulations $M$ increases.
Furthermore, Table \ref{TableIII} shows that there is no significant
difference between the estimates obtained from sampling the exact
solution $X_{\tau _{n}}$ and Scheme \ref{scheme:MM1}, even though $
\mathbb{E} \left(\arctan \left(1+(X_{\tau _{n}}^{l})^{2} \right)
\right)$ involves the computation of high order moments of $X_{\tau
_{n}}$.

The above simulation results illustrate the feasibility of Scheme
\ref{scheme:MM1} for approximating functionals of linear SDEs with
multiplicative noise.  At this point is worth to mention that, with
the uniform time partition consider here, the Euler scheme leads
divergent results or computer overflows in the integration of the
equation (\ref{Eq Ej1}).

\begin{example}
\textbf{Nonautonomous nonlinear SDE }\cite{Talay90b}.%
\begin{equation}
d\left[
\begin{array}{c}
X_{t}^{1} \\
X_{t}^{2}%
\end{array}%
\right] =\left[
\begin{array}{c}
-X_{t}^{2} \\
X_{t}^{1}%
\end{array}%
\right] dt+\left[
\begin{array}{c}
0 \\
\frac{\sin (X_{t}^{1}+X_{t}^{2})}{\sqrt{1+t}}%
\end{array}%
\right] dW_{t}^{1}+\left[
\begin{array}{c}
\frac{\cos (X_{t}^{1}+X_{t}^{2})}{\sqrt{1+t}} \\
0%
\end{array}%
\right] dW_{t}^{2},  \label{Eq Ej2}
\end{equation}%
with initial condition $(X_{0}^{1},X_{0}^{2})=(1,1)$ and $t\in \lbrack 0,10]$%
. For this equation, $E(\phi (X_{t}))=\left\vert
X_{t_{0}}\right\vert ^{2}+\log (1+t)$, with $\phi (X)=\left\vert
X\right\vert ^{2}$.
\end{example}

\begin{figure}[h]
\centering
\includegraphics[width=6.5in]{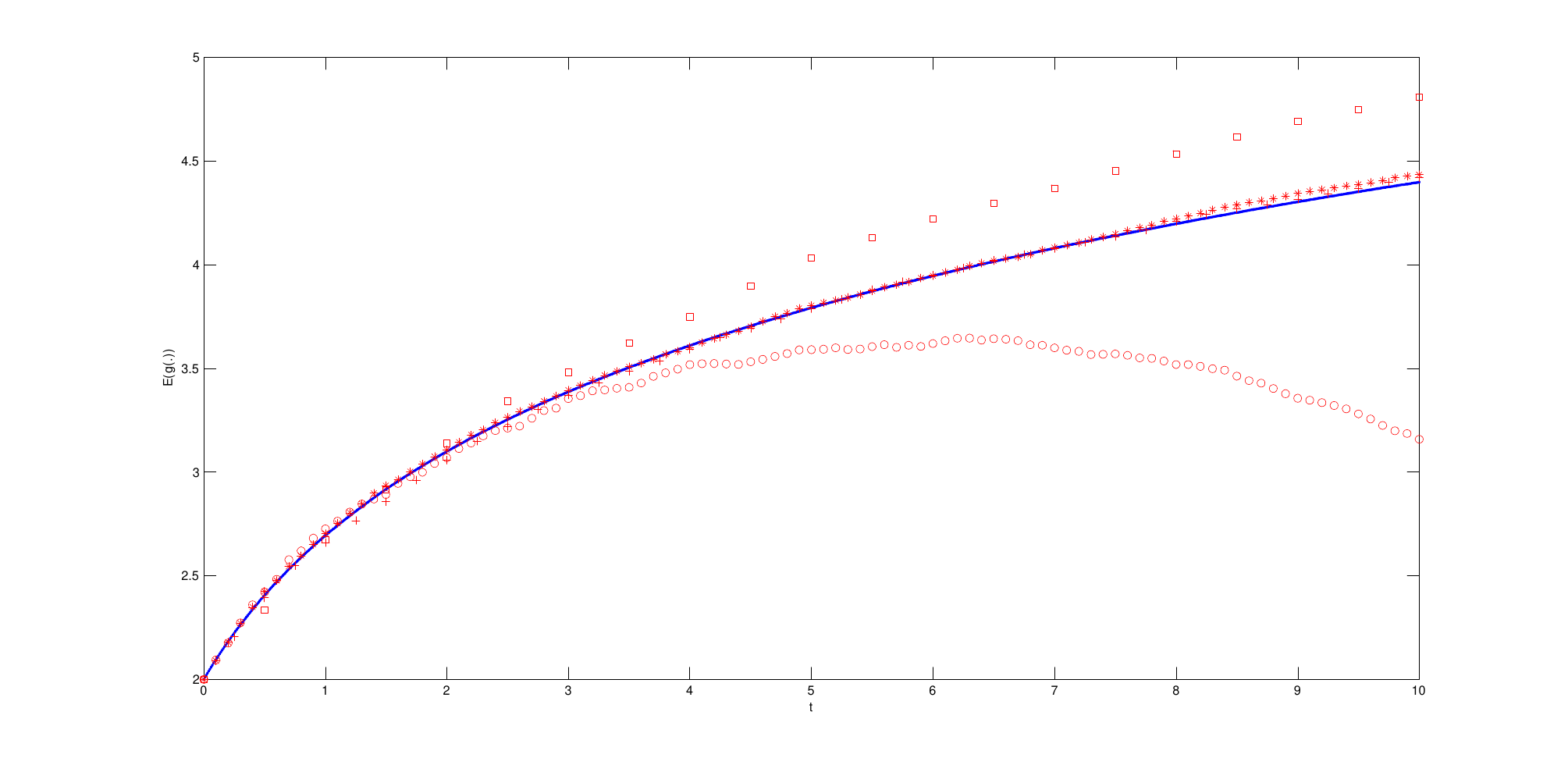}
\caption{Integration of Example 2. Exact Value: solid line. Scheme
\protect
\ref{scheme:MM1}: $\square $ with $\Delta =0.5$, $+$ with $\Delta =0.25$, $%
\ast $ with $\Delta =0.1$. Euler with Romberg extrapolation: $o$ with $%
\Delta =0.05$ and $\Delta =0.1$} \label{Fig2}
\end{figure}

It is well-known from \cite{Talay90b} that via Monte Carlo
simulations: 1) both, the Euler and the Milstein schemes with fixed
stepsize $\Delta =0.01$ fail to approximate $E(\phi (X_{t}))$; and
2) the second order method arising from Romberg's extrapolation of
the Euler scheme with stepsizes $0.02 $ and $0.01$ gives a
satisfactory approximation to $E\left( \phi \left( X_{t}\right)
\right) $, but fails when the stepsizes are $0.05$ and $0.1$.
Similarly to the fourth figure in \cite{Talay90b}, Figure 2
illustrates this last result for a Monte Carlo estimation with
$M=10000$ simulations.

 Figure 2 also shows the computation of
$E(\phi (X_{t}))$ via Monte Carlo method and Scheme
\ref{scheme:MM1}, but on uniform time partitions with stepsizes
$\Delta =0.5,0.25,0.1$ and $M=10000$ simulations. In addition, Table
\ref{TableIV} provides the estimates $\widehat{e}$ of the mean
errors $e=E(\phi (z_{N}))-E(\phi (X_{T}))$ resulting from the
integration of (\ref{Eq Ej2}) via Scheme \ref{scheme:MM1} with
different stepsizes. For this, the simulated trajectories
$z_{N}^{\{i,j\}}$, $i=1,...,K $ and $j=1,...,M$, were are arranged
into $K=100$ batches of $M=10000$ trajectories each for computing
\begin{equation*}
\text{\
}\widehat{e}=\frac{1}{K}\sum\limits_{j=1}^{K}\widehat{e}_{j}\text{ \
\ \ \ \ \ with \ \ \ \ \ }\widehat{e}_{j}=\frac{1}{M}\sum\limits_{i=1}^{M}%
\phi \left( z_{N}^{\{i,j\}}\right) -E\left( \phi \left( X_{T}\right)
\right) .\text{ }
\end{equation*}%
The $90\%=100(1-\alpha )\%$ confidence interval of the Student's $t$
distribution with $K-1$ degrees for the mean error is given by
\begin{equation*}
\lbrack \widehat{e}-\Delta \widehat{e},\widehat{e}+\Delta
\widehat{e}],
\end{equation*}%
where
\begin{equation*}
\Delta \widehat{e}=t_{1-\alpha ,K-1}\sqrt{\frac{\widehat{\sigma }_{e}^{2}}{K}%
},\text{ \ \ \ \ with\ \ \ \ \ \ \ \ \ }\widehat{\sigma }_{e}^{2}=\frac{1}{%
K-1}\sum\limits_{j=1}^{K}(e_{j}-\widehat{e})^{2}.
\end{equation*}%
For comparison, the same estimate of the mean error for Euler scheme
is also given in Table \ref{TableIV}. This illustrates again the
better performance of the Scheme \ref{scheme:MM1} introduced in this
paper.\newline

\begin{table}[bt]
\begin{center}
{\footnotesize
\begin{tabular}{|c||c|c|c|c|}
\hline $\widehat{e}/\Delta $ & $1$ & $0.5$ & $0.25$ & $0.1$ \\
\hline\hline Scheme \ref{scheme:MM1} &
\multicolumn{1}{|r|}{$-2.2360\pm 0.0093$} &
\multicolumn{1}{|r|}{$-0.4512\pm 0.0067$} &
\multicolumn{1}{|r|}{$-0.0868\pm 0.0054$} &
\multicolumn{1}{|r|}{$0.0076\pm 0.0053$} \\ \hline
Euler & \multicolumn{1}{|r|}{$-2435.8\pm 1.7826$} & \multicolumn{1}{|r|}{$%
-235.05\pm 0.2192$} & \multicolumn{1}{|r|}{$-32.031\pm 0.0361$} &
\multicolumn{1}{|r|}{$-5.7704\pm 0.0101$} \\ \hline
\end{tabular}
}
\end{center}
\caption{{\protect\small Estimate }$\widehat{e}${\protect\small \ of
the mean error }$E(\protect\phi (z_{N}))-E(\protect\phi
(X_{T}))${\protect\small \ in the integration of (\protect\ref{Eq
Ej2}) by means of Scheme \protect \ref{scheme:MM1} and the Euler
scheme for different integration stepsizes }$ \Delta $. }
\label{TableIV}
\end{table}

\section{Conclusions}

A weak Local Linearization scheme for stochastic differential
equations with multiplicative noise was introduced. The scheme
preserves the first two moments of the solution of linear SDEs and
the mean square stability that such solution may have. The order-1
of weak convergence was proved and the practical performance of the
scheme in the evaluation of functionals of linear and nonlinear SDEs
was illustrated with numerical simulations. The simulations also
showed the significant higher accuracy of the introduced scheme in
comparison with the Euler scheme. \newline

\textbf{Acknowledgement.} This work was partially supported by
FONDECYT Grant 1140411. CMM was also partially supported by BASAL
Grant  PFB-03.

\end{document}